\documentclass[10pt,a4paper]{article}
\usepackage[utf8]{inputenc}
\usepackage{amsmath}
\usepackage{amsfonts}
\usepackage{amssymb}
\usepackage{graphicx}

\usepackage{amsthm}
\theoremstyle{plain}
\usepackage{tikz}
\usetikzlibrary{cd}
\usepackage{mathtools}

\usepackage{hyperref}

\newtheorem{thm}{Theorem}
\newtheorem{definition}[thm]{Definition}
\newtheorem{cor}[thm]{Corollary}
\newtheorem{lem}[thm]{Lemma}
\newtheorem{rem}[thm]{Remark}
\newtheorem{prop}[thm]{Proposition}
\newtheorem{ex}[thm]{Example}

\newtheorem{thmintro}{Theorem}

\newcommand{\T}{\mathbb{T}}
\newcommand{\R}{\mathbb{R}}
\newcommand{\C}{\mathbb{C}}
\newcommand{\Z}{\mathbb{Z}}
\newcommand{\A}{\mathbb{A}}
\newcommand{\G}{\mathbb{G}}
\newcommand{\N}{\mathbb{N}}

\newcommand{\cA}{\mathcal{A}}

\newcommand{\Oh}{\mathcal{O}}

\newcommand{\V}{\mathcal{V}}
\newcommand{\W}{\mathcal{W}}

\newcommand{\del}{\partial}
\newcommand{\delbar}{\overline{\partial}}

\newcommand{\coker}{\operatorname{coker}}
\newcommand{\pr}{\operatorname{pr}}

\newcommand{\Id}{\operatorname{Id}}
\newcommand{\Pro}{\mathbb{P}}

\DeclareMathOperator{\an}{an}
\DeclareMathOperator{\gr}{gr}

\DeclareMathOperator{\Fil}{Fil}

\DeclareMathOperator{\Gra}{Gra}
\DeclareMathOperator{\Rs}{Rs}
\DeclareMathOperator{\Bun}{Bun}
\DeclareMathOperator{\Rep}{Rep}

\DeclareMathOperator{\End}{End}
\DeclareMathOperator{\Hom}{Hom}

\DeclareMathOperator{\Spec}{Spec}

\mathtoolsset{centercolon=true}
\newcommand{\Cdot}{{\raisebox{-0.7ex}[0pt][0pt]{\scalebox{2.0}{$\cdot$}}}}
\mathchardef\mhyphen="2D

\let\oldabstract\abstract
\let\oldendabstract\endabstract
\makeatletter
\renewenvironment{abstract}
{%
	{\list{}{\addtolength{\leftmargin}{2em} % change this value to add or remove length to the the default
			\listparindent 0em%
			\itemindent    \listparindent%
			\rightmargin   \leftmargin%
			\parsep        \z@ \@plus\p@}%
		\item\relax}%
	{\endlist}%
	\oldabstract}
{\oldendabstract}
\makeatother

\usepackage{authblk}
\title{Toric Vector Bundles: GAGA and Hodge Theory}
\author{Jonas Stelzig}

\begin{document}
\maketitle
\begin{abstract}
We prove a GAGA-style result for toric vector bundles with smooth base and give an algebraic construction of the Fr\"olicher approximating vector bundle that has recently been introduced by Dan Popovici using analytic techniques.
\end{abstract}

\section{Introduction}

A toric variety over a field $k$ is an algebraic variety $X$ over $k$ with a $\G_m^n$-action that has a dense open orbit on which the group acts simply transitively. A vector bundle on such $X$ is called toric if it is equipped with a $\G_m^n$-action s.t. the projection is an equivariant map.\\

Toric varieties and vector bundles are an important source of examples in algebraic geometry. Just as normal toric varieties can be studied by combinatorial data, toric vector bundles (and also more general classes of equivariant sheaves) on a given normal toric variety $X$ have been classified in terms of linear-algebra-data (roughly as vector spaces and filtrations with certain compatibility conditions), c.f. \cite{kaneyama_equivariant_1975}, \cite{kaneyama_torus-equivariant_1988}, \cite{klyachko_equivariant_1990}, \cite{klyachko_vector_1991}, \cite{perling_graded_2004}.\\

If $k=\C$, for every toric variety $X$ over $\C$ one also has a natural notion of holomorphic toric vector bundles over $X^{\an}$, the latter meaning (the set of complex points of) $X$ seen as a complex analytic space. One obtains an analytification functor:
\[
\text{toric vector bundles on }X\longrightarrow\text{ holomorphic toric vector bundles on }X^{\an}
\]
The first main result of this article is that for smooth toric varieties, this functor is an equivalence of categories:

\begin{thmintro}
For a smooth toric variety $X$ over $\C$, analytification induces an equivalence of categories between algebraic toric vector bundles on $X$ and holomorphic toric vector bundles on $X^{\an}$. The same is true for toric vector bundles with (flat or arbitrary) equivariant connections.
\end{thmintro}

Of course, there is the known GAGA-principle by Serre, asserting an equivalence of the categories of coherent sheaves on a complex projective variety and its analytification. The above Theorem is not a formal consequence of this. For instance, $X$ is not assumed to be projective and even for projective $X$, the equivariant structure is sheaf-theoretically described by an isomorphism of sheaves over $\G_m^n\times X$, which is not projective.\\

We do not need the full classification of toric vector bundles. In fact, it is enough for our purposes to consider algebraic toric vector bundles on affine spaces and we include a brief but largely self-contained treatment of these. A key notion is the Rees-bundle construction, which associates to (suitable) multifiltered vector spaces $(V,F_1,...,F_n)$ a toric vector bundle $\xi_{\A^n}(V,F_1,...,F_n)$ on $\A^n$.\\

Toric vector bundles have also been studied in connection with Hodge theory, see e.g. \cite{deninger_$gamma$-factors_1991}, \cite{simpson_hodge_1997}, \cite{simpson_mixed_1997}, \cite{deninger_$gamma$-factors_2001}, \cite{penacchio_structures_2003}, \cite{goncharov_hodge_2008}, \cite{kapranov_real_2012}, \cite{penacchio_mixed_2011}, in part apparently independent of the classification. A basic idea is that, since a Hodge structure is a multifiltered vector space, one may apply Rees-bundle like constructions to it and obtain a toric vector bundle.\\

Recently in \cite{popovici_adiabatic_2019}, D. Popovici introduced the so-called Fr\"olicher approximating bundle, in order to give a more conceptual proof of his earlier result that deformation limits of Moishezon manifolds are again Moishezon. This construction associates to every compact complex manifold $X$ and every integer $k$ a holomorphic vector bundle $\cA^k$ on $\C$ (thus necessarily trivial) which interpolates between the de-Rham cohomology and the degenerating page of the Fr\"olicher spectral sequence. I.e., it comes equipped with distinguished isomorphisms $\cA^k(h)\cong H_{dR}^k(X,\C)$ for $h\neq 0$ and $\cA^k(0)\cong\bigoplus_{p+q=k}E_r^{p,q}$. The construction is analytic, in particular it involves choosing a metric and introducing Laplace-type pseudo-differential operators for the higher pages of the Fr\"olicher spectral sequence. We show here that this bundle can be seen as a special case of the Rees-bundle construction, thereby giving a purely algebraic description.

\begin{thmintro}
For $X$ a compact complex manifold and $\cA^k$ the $k$-th Fr\"olicher approximating bundle, there is a canonical isomorphism
\[
\cA^k\cong\xi_{\A^1}(H_{dR}^k(X,\C),F)^{\an},
\]
where $F$ denotes the Hodge-filtration.
\end{thmintro}

\section{Preliminaries}

This section collects some basic notions and result from more general theory which are needed later on. No claim to originality is made here: All results (except possibly for mistakes on the author's behalf) are contained in one or more of \cite{kaneyama_equivariant_1975}, \cite{kaneyama_torus-equivariant_1988}, \cite{klyachko_equivariant_1990}, \cite{klyachko_vector_1991}, \cite{perling_graded_2004}, \cite{deninger_$gamma$-factors_1991}, \cite{simpson_hodge_1997}, \cite{simpson_mixed_1997}, \cite{deninger_$gamma$-factors_2001}, \cite{penacchio_structures_2003}, \cite{goncharov_hodge_2008}, \cite{kapranov_real_2012}, \cite{penacchio_mixed_2011}.
\subsection{Definitions}
Throughout all the section, we fix an algebraically closed field $k$ of characteristic zero and denote by $G$ an algebraic group over $k$, which will soon be set to be the algebraic $n$-torus $\G_m^n$. We first repeat some standard definitions. Even though everything is formulated for algebraic varieties, the reader may note that, mutatis mutandis, most of the definitions make sense in the holomorphic, smooth or even continuous setting.
\begin{definition}
A $G$-variety is an algebraic variety $X$ over $k$ together with an action $\rho_X:G\times X\longrightarrow X$ of the group $G$. If $G=\G_m^n$ and there is a dense open orbit on which it acts simply transitively, $(X,\rho_X)$ is called a \textbf{toric variety}. 
\end{definition}
\begin{ex}
Any variety is a $G$-variety for the trivial action. The varieties $\G^m\times \A^{n-m}$ for the natural action of $G$ by multiplication are toric. 
\end{ex}

\begin{definition}
An \textbf{equivariant sheaf} (of $\Oh_X$-modules) on a $G$-variety $(X,\rho_X)$ is tuple $(\V,\Phi)$, where $\V$ is a sheaf and $\Phi$ is an isomorphism
\[
\Phi:\rho_X^*\V\longrightarrow \pr_2^*\V
\]
of $\Oh_{G\times X}$-modules that satisfies the cocycle condition
\[
\pr_{23}^*\Phi\circ (\Id_G\times\rho_X)^*\Phi=(\mu_G\times\Id_X)^*\Phi,\label{cocycle-condition}
\]
where $\pr_{23},\Id_G\times\rho_X,\mu_G\times\Id_X$ are maps $G\times G\times X\longrightarrow G\times X$ given by projection, action and multiplication respectively.
\end{definition}
One may check that if $\V$ is a coherent and locally free, i.e. the sections of a vector bundle $E\longrightarrow X$, this definition is equivalent to requiring $E$ to be a $G$-variety s.t. the action commutes with the projection.
\begin{ex}
Let $\Omega_X$ be the sheaf of K\"ahler differentials. Recall that for any product $X\times Y$ there is a canonical identification $\Omega_{X\times Y}\cong \pr_X^*\Omega_X\oplus\pr_Y^*\Omega_Y$. As a consequence one may equip $\Omega_{X}$ with an action via
\[
\Phi_{\Omega}:\rho_X^*\Omega_X\longrightarrow\Omega_{G\times X}\cong\pr_G^*\Omega_G\oplus\pr_X^*\Omega_X\longrightarrow\pr_X^*\Omega_X,
\]
where the first map is pullback via $\rho_X$ and the second map is projection.
\end{ex}

\begin{definition}
An \textbf{equivariant connection} on an equivariant sheaf $(\V,\Phi)$ on a toric variety $(X,\rho_X)$ is a connection $\nabla:\V\longrightarrow\V\otimes\Omega_X$ such that the following diagram commutes:

\[
\begin{tikzcd}
\V\ar[rrr,"\nabla"]\ar[d,"\Phi^{ad}"]&&&\V\otimes\Omega_{X}\ar[d,"\Phi^{ad}\otimes\Phi_{\Omega_{X}}^{ad}"]\\
{\rho_X}_*\pr_X^*\V\ar[rrr,"{\rho_X}_*(\pr_X^*\!\!\nabla)_X"]&&&{\rho_X}_*(\pr_X^*\V\otimes\pr_X^*\Omega_{X})
\end{tikzcd}
\]
Here, $(\pr_X^*\!\!\nabla)_X$ denotes the composite of $\pr^*_X\nabla:\pr_X^*\V\longrightarrow\pr_X^*\V\otimes \Omega_{G\times X}$ with $\Id\times{res}$ where $res:\Omega_{G\times X}\longrightarrow \pr_X^*\Omega_X$ is projection.
\end{definition}

On affine schemes, the above notions can, as usual, be translated into commutative algebra and it is this description that will be used later on in this section.
\begin{prop}
The global sections functor yields an equivalence of categories between:
\begin{enumerate}
\item Quasi-coherent equivariant sheaves on an affine toric variety $X=\Spec A$ and $n$-graded $A$-modules.
\item Equivariant vector bundles with connection on a toric variety $X=\Spec A$ and $n$-graded $A$ modules $M$ which are locally free of finite rank and equipped with a map
\[
\nabla:M\longrightarrow M\otimes \Omega_A,
\]
that satisfies the Leibniz rule and respects the grading. Here, $\Omega_A$ denotes the module of K\"ahler differentials, the generators of which satisfy $\deg(dX_i)=\deg(X_i)$, and the right hand side has the tensor product grading.
\end{enumerate}
\end{prop}
The proof is standard and we omit it. Let us just describe how the $n$-grading on $A$ is defined: The action $\rho_X$ induces a coaction $\rho^*_X:A\longrightarrow A\otimes k[z_1^{\pm 1},...,z_n^{\pm 1}]$ and for any multiindex $p=(p_1,...,p_n)\in \Z^n$, one sets 
\[A^p:=\{a\in A\mid \rho^*(a)=a\otimes z_q^{p_1}\cdot...\cdot z_n^{p_n}\}\]

\subsection{Toric vector bundles on affine and projective spaces}
It is well-known that one can further simplify the commutative algebra from the previous section and describe equivariant sheaves on normal toric varieties in terms of linear algebra. We keep notations from the previous subsection and recall how this works for the case $X=\A^n=\Spec k[z_1,...,z_n]$.\\

\textbf{Notational conventions:} We fix some natural number $n$ and denote by for any toric variety $X$ by $\Bun(X,\G_m^n)$ the category of coherent equivariant sheaves on $X$. By $\Fil^n_k$, we denote the category of finite dimensional $k$-vector spaces $(V,F_1,...,F_n)$ with $n$ separated and exhaustive descending filtrations (i.e. $F_i^P=\{0\}$ and $F_i^p=V$ for $P\gg 0 \gg p$). Given two filtered vector spaces $(V,F), (V',F')$, the tensor product $V\otimes_kV'$ is equipped with the filtration $(F\otimes F')^\Cdot=\sum_{p+q=\Cdot} F^p\otimes_k {F'}^q$ and this induces a tensor product on $\Fil^n_k$. We sometimes write simply $V$ instead of $(V,F_1,...,F_n)$ for an object in $\Fil_k^n$.\\
We write $A:=k[z_1,...,z_n]$ and $B:=k[z_1^{\pm 1},...,z_n^{\pm 1}]$. We view these as equipped with the standard $n$-grading, i.e., $z_i^p$ has $i$-th degree $p$ and zero else. For a multiindex $p=(p_1,...,p_n)\in\Z^n$, we write $|p|:=\sum_{i=1}^n p_i$ and $z^p:=z_1^{p_1}\cdot...\cdot z_n^{p_n}\in B$. For $n$ filtrations $F_i$ on some vector space, we set $F^p:= F_1^{p_1}\cap...\cap F_n^{p_n}$. Given some $\lambda\in\Z$, we use the notation $\lambda p:=(\lambda p_1,...,\lambda p_n)$ and $p\pm_i \lambda:= (p_1,...,p_{i-1},p_i\pm \lambda,p_{i+1},...,p_n)\in\Z^n$. We abbreviate $(0,...,0)$ to just $0$.\\ 
and similarly with $\geq_i$ instead of $\pm_i$. Also, for $p,r\in\Z^n$, we write $p\geq r:\Leftrightarrow p_i\geq r_i~\forall i\in\{1,...,n\}$.\\

\textbf{From filtrations to sheaves:}\\ 
Starting from a multi-filtered vector space $(V,F_1,...,F_n)$, denote by 
\[
\Rs^n(V):=F^{0}(V\otimes_k B)=\sum_{p\in\Z^n}F^pV\otimes_k z^{-p}A\subseteq V\otimes_k B
\]
the \textbf{Rees-module} associated with $V$. Since it is a submodule of $V\otimes B$, it is always torsion free. Define the (algebraic) \textbf{Rees-sheaf} $\xi_{\A^n}(V)$ to be the coherent sheaf associated to this module, with $\G_m^n$-equivariant structure corresponding to the grading. This construction is functorial.\\

It turns out that every toric vector bundle is of the form $\xi_{\A^n}(V)$. To formulate this more precisely, recall that a splitting for a set of filtrations $F_1,...,F_n$ on a vector space $V$ is a decomposition 
\[
V=\bigoplus_{p\in\Z^n} V^p\text{ s.t. }F^r_iV=\bigoplus_{\substack{p\in\Z^n\\p_i\geq r}}V^p.\] 
For $n=1,2$, one can always construct a splitting by choosing appropriate bases, but three or more filtrations may or may not be splittable, as may be seen by examples constructed from $n$ lines in $k^2$. Denoting by $\Fil_k^{n,splittable}$ the subcategory of $\Fil^{n}_k$ consisting of those multifiltered vector spaces that admit a splitting, one has:

\begin{thm}\label{equivalence sheaves and filtrations}
	The functor $\xi_{\A^n}$ induces an equivalence of categories

	\[\begin{tikzcd}
	\Bun(\A^n,\G_m^n)\ar[r,leftrightarrow]&\Fil^{n,splittable}_k,
	\end{tikzcd}
	\]
	which is compatible with direct sums and tensor products and where maps with constant rank on the left side correspond to maps that are $n$-strict on the right hand side. For $n=1,2$ one has $\Fil^n_k=\Fil_k^{n,splittable}$.
\end{thm}
	
We will not reproduce a proof here, but just explain the condition on $\xi_{\A^n}(V)$ to be a vector bundle and the notion of $n$-strictness used in the statement. For this, we need a small Lemma that can be checked directly from the definition of the Rees-module.

\begin{lem} \label{Restriction properties of the affine Rees bundle}
	The functors $\xi_{\A^n}(\_)$ behave as follows when restricted to torus-invariant subsets:
	
	\begin{enumerate}
		\item There is a functorial, $\G_m^n$-equivariant isomorphism \[\xi_{\A^n}(V,F_1,...,F_n)|_{\G_m\times\A^{n-1}}\cong \pr_{\A^{n-1}}^*\xi_{\A^{n-1}}(V,F_2,...,F_n)\] and similarly for other subvarieties of the form $\A^r\times\G_m^s\times\A^t$. In particular, there is an equivariant isomorphism $\xi_{\A^n}(V)|_{\G_m^n}\cong \Oh_{\G^n_m}\otimes_k V$.
		\item Given $(V,F_1,...,F_n)$ in $\Fil^n_k$, set $D^p:=\frac{F^p}{\sum_{i=1}^n F^{p+_i1}}$. There is a functorial isomorphism, equivariant for the $\G_m^n$-action given by the restricted action on the left and by the grading on the right.
		\[
		\xi_{\A^n}(V)|_{\{0\}}\cong \bigoplus_{p\in\Z^n} D^{-p}
		\]
	\end{enumerate}
\end{lem}

How to see when $\xi_{\A^n}(V)$ is a vector bundle? For example this is the case if $V=k$ so that each filtration has a single jumping index $r_i$, i.e. $F_i^{r_i+1}=0\subseteq F_i^{r_i}=k$. Denoting $r=(r_1,...,r_n)$, the Rees module is then just $z^{-r}k[z_1,...,z_n]$, which is a free module of rank one. Let us denote the resulting Rees-bundle by $\Oh_{\A^n}(r)$.

\begin{prop}\label{characterisation splittable filtrations}
	Let $(V,F_1,...,F_n)$ be in $\Fil^n_k$. The following assertions are equivalent:
	\begin{enumerate}
		\item The filtrations $F_1,...,F_n$ are splittable.
		\item $\xi_{\A^n}(V,F_1,...,F_n)$ is a vector bundle.
		\item The spaces $D^p:=\frac{F^p}{\sum_{i=1}^n F^{p+_i1}}$ satisfy $\sum_{p\in\Z^n}\dim D^p=\dim V$.
	\end{enumerate}
If one of these conditions is satisfied, any splitting $V=\bigoplus_{p\in\Z^n} V^p$ determines an isomorphism
	\[
	\xi_{\A^n}(V)\cong \bigoplus_{p\in\Z^n} V^p\otimes_k\Oh_{\A^n}(p)
	\]
\end{prop}
\begin{proof}
	We first show $1.\Rightarrow 2.$: If $V=\bigoplus V^p$ is a splitting of the $F_i$, there is an isomorphism $\xi_{\A^n}(V)=\bigoplus \xi_{\A^n}(V^p)\cong \bigoplus V^p\otimes\Oh_{\A^n}(p)$, which is free.\\
	
	From $2.$ to $3.$, note that a vector bundle has constant fibre dimension. By Lemma \ref{Restriction properties of the affine Rees bundle}, the fibre over any point (e.g. $(1,...,1)$) in $\G_m^n$ is canonically identified with $V$, while the fibre over $0$ is canonically identified with $\bigoplus D^p$.\\
	
	Finally, assuming $3.$, choose, for every $p\in\Z^n$, subspaces $V^{p}\subseteq F^{p}$ s.t. $V^{p}$ projects isomorphically onto $D^{p}$.
	By construction, we have 
	\[
	F^{p}=\sum_{\substack{r\in\Z^n\\r\geq p}}V^{r}.
	\]
	In particular, 
	\[
	\dim V\leq\sum_{p\in\Z^n}\dim V^{p}=\sum_{p\in\Z^n}\dim D^{p}=\dim V.
	\]
	Therefore, there has to be an equality and the sum of the $V^{p}$ is direct.
\end{proof}

\begin{definition} \label{r-strictness}
A map  
	\[
	f:(V,F_1,...,F_n)\longrightarrow(W,G_1,...,G_n)
	\]
	in $\Fil^n_k$ is called $r$-strict if for every collection $\{i_1,...,i_r\}$ of indices in $\{1,...,n\}$ and all $r$-tuples of integers $(p_{1},...,p_{r})\in\Z^r$,
	\[
	f(F_{i_1}^{p_1}\cap...\cap F^{p_r}_{i_r})=G_{i_1}^{p_r}\cap...\cap G^{p_r}_{i_r}\cap\operatorname{im}(f).
	\]
\end{definition}
	E.g., $1$-strictness coincides with usual strictness for all filtrations and if there exist splittings for the $F_i$ and $G_i$ respected by $f$, then $f$ is $n$-strict.

\begin{prop}\label{maps between filtered spaces vs. maps between Rees-bundles}
	Let $f:V\longrightarrow W$ be a map in $\Fil_k^n$ and $\xi_{\A^n}(f)$ the associated map of Rees-sheaves. 
	\begin{itemize}
		\item The morphism induced by the inclusion $\ker f\hookrightarrow \underline{V}$ gives a canonical identification
		\[
		\xi_{\A^n}(\ker f)\cong\ker\xi_{\A^n}(f).
		\]
		\item There is an exact sequence
		\[
		0\longrightarrow T\longrightarrow\coker\xi_{\A^n}(f)\overset{\varphi}{\longrightarrow}\xi_{\A^n}(\coker f)\longrightarrow 0,
		\]
		where $\coker\xi_{\A^n}(f)$ is the sheaf-theoretic cokernel and $T$ is the torsion subsheaf. Further, \[
	f \text{ is } r \text{-strict}\Leftrightarrow\operatorname{codim}(\operatorname{supp}(T))> r.
	\]
	\end{itemize}
\end{prop}

\begin{proof}
	It suffices to check this on global sections, i.e. on the Rees-module. For the first point, note that for any $p\in\Z^n$, a section $v_p\otimes z^p$ with $v_p\in F^p_V$ is mapped to zero iff $v_p\in\ker f$.\\
	
	For the second point, let $\pi:W\rightarrow \coker f$ denote the projection. The global sections of the two right members of the claimed sequence are then given (where the sum is direct as $k$-vector spaces) by
	\[
	\Gamma(\A^n,\coker\xi_{\A^n}(f))=\bigoplus_{p\in\Z^n}\frac{F_W^p}{f(F_V^p)}z^{-p}
	\]
	and
	\[
	\Gamma(\A^n,\xi_{\A^n}(\coker f))=\bigoplus_{p\in\Z^n}\pi F_W^pz^{-p}.
	\]
	Denoting by $\varphi^p:\frac{F_W^p}{f(F_V^p)}\longrightarrow \pi F_W^p$ the natural map, we can define the map $\varphi$ to be induced by the direct sum of the $\varphi^p$. Using 
	\[\ker\varphi^p=\{w\in F_W^p\mid \exists q\in\N^n\text{ s.t. } w\in f(F_W^{p-q})\}\text{ mod }f(F_V^p),
	\]
	one verifies that $\ker\varphi=T$ coincides with the torsion subsheaf.\\
	
	For simplicity, we verify the statement on the dimension of the support of $T$ only in the case $r=n$: In fact, $n$-strictness is equivalent to the condition that all the $\varphi^p$ are isomorphisms, which is in turn equivalent to $\varphi$ being an isomorphism.
\end{proof}

Let us end with describing a projective variant of the above constructions. For $(V,F_0,...,F_n)$ in $\Fil^{n+1}_k$, define $\xi_{\Pro^n_k}(V)$ as the invariants under the diagonal $\G_m$ of the pushforward of $\xi_{\A^{n+1}}(V)|_{\A^{n+1}\setminus \{0\}}$. Alternatively, apply $\xi_{\A^n}$ to $V$ equipped with all filtrations except one, say $F_i$ to obtain a bundle on $\A^n\cong \{Z_i\neq 0\}\subseteq \Pro_k^n$. These patch together by Lemma \ref{Restriction properties of the affine Rees bundle}.
\begin{cor}
	The functor $\xi_{\Pro^n}$ gives rise to an equivalence of categories
	\[\begin{tikzcd}
	\Bun(\Pro^n,\G_m^n)\ar[r,leftrightarrow]&\Fil^{n+1,n-splittable}_k,
	\end{tikzcd}
	\]
	where $\Fil^{n+1,n-splittable}$ is the full subcategory of $\Fil^{n+1}_k$ consisting of objects s.t. every $n$ out of the $n+1$ filtrations admit a splitting. 
\end{cor}

\section{Toric GAGA}
In this section, we are interested in the case $k=\C$ and we only treat vector bundles, not more general equivariant sheaves. We consider equivariant connections and weaker (than being algebraic) regularity conditions on transition functions and group actions.\\

\textbf{Notations and conventions:} Let $\Gra^n_\C$ denote the category of finite dimensional $n$-graded $\C$-vector spaces. For $G=(\C^\times)^n$ or $G=\T^n:=(S^1)^n$, by $\Bun^\omega(\C^n,G)$ with $\omega\in\{\operatorname{alg}, \operatorname{hol},  \operatorname{sm}, \operatorname{cont}\}$, we mean algebraic, holomorphic, smooth or continuous $G$-equivariant bundles on $\C^n$ when meaningful. We identify $\Bun(\A^n,\G_m^n)$ from the previous section with $\Bun^{\operatorname{alg}}(\C^n,(\C^\times)^n)$ (by considering the complex valued points) and we switch freely between geometric vector bundles and locally free sheaves over the corresponding structure sheaf. By $\Rep^\omega(G)$, we denote algebraic, holomorphic, smooth or continuous representations of $G$. Those cases of interest to us are related to each other by the following diagram:

\[\tag{$\ast$}
\begin{tikzcd}
\Bun^{\operatorname{\operatorname{alg}}}(\C^n,(\C^\times)^n)\ar[r]\ar[dd]&\Rep^{\operatorname{alg}}((\C^\times)^n)\ar[dd]\ar[rddd]&\\
&&\\
\Bun^{\operatorname{hol}}(\C^n,(\C^\times)^n)\ar[r]\ar[dd]&\Rep^{\operatorname{hol}}((\C^\times)^n)\ar[dd]\ar[rd]&\\
&&\Gra^n_\C\\
\Bun^{\operatorname{sm}}(\C^n,\T^n)\ar[r]\ar[dd]&\Rep^{\operatorname{sm}}(\T^n)\ar[dd]\ar[ru]&\\
&&\\
\Bun^{\operatorname{cont}}(\C^n,\T^n)\ar[r]&\Rep^{\operatorname{cont}}(\T^n)\ar[ruuu]&
\end{tikzcd}
\]
Here, the vertical arrows forget some structure (or in sheaf-theoretic terms, tensor by a bigger structure sheaf and restrict the action), the horizontal ones are restriction to the fibre at $0$ and the diagonal ones are given by the rule $(V,\rho)\mapsto V=\bigoplus_{p\in\Z^n} V^p$ where $V^p$ with $p=(p_1,...,p_n)$ is the eigenspace of the character 
\[
\chi_p:(\lambda_1,...,\lambda_n)\mapsto \lambda^{-p}:=\lambda_1^{-p_1}\cdot...\cdot\lambda_n^{-p_n}.
\]

All diagonal arrows are equivalences of categories, hence so are all arrows in the middle column. To a $G$-representation $(V,\rho)$ in $\Rep^\omega(G)$, one can associate the trivial (geometric) vector bundle $\widetilde{V}:=\C^n\times V$ with product action
\[
\lambda.(x,v)=(\lambda.x,\rho(\lambda)x).
\] 
We also denote this $\widetilde{V}^\omega$ if we want to emphasize that we consider it as a bundle in $\Bun^\omega(\C^n,G)$. E.g., in the notation of the previous section, $\widetilde{V}^{\operatorname{alg}}=\bigoplus_{p\in\Z^n}V^p\otimes_\C\Oh_{\A^n}$, where as above $V^p$ is the eigenspace of $\chi_p$, but equipped with the trivial action.\\

The invariant global sections of $\widetilde{V}^\omega$ can be identified with equivariant (algebraic, holomorphic, smooth or continuous) maps $\C^n\longrightarrow V$ and the restriction of $\widetilde{V}$ to $0$ is canonically isomorphic to $(V,\rho)$. In particular, all horizontal arrows in $(\ast)$ are essentially surjective.\\

Because $\C^n$ is $\T^n$-equivariantly contractible to the fixed point $0$, the restriction to the fibre
\[
\Bun^{\omega}(\C^n,\T^n)\longrightarrow \Rep^{\omega}(\T^n)
\]
induces a bijection on isomorphism classes\footnote{This follows from the `homotopy invariance of isomorphism classes of equivariant vector bundles': Given a compact topological group $G$ and two equivariantly homotopic equivariant maps ${f_0,f_1:X\longrightarrow Y}$ between $G$-spaces with $X$ paracompact. Then for any vector bundle $\V$ on $Y$ the pullbacks $f_0^*\V$ and $f_1^*\V$ are equivariantly isomorphic.\\
	This is essentially proven in \cite{atiyah_$k$-theory_1967}, p. 40f. However, there only the case of a compact hausdorff base and a finite group is treated. See \cite{segal_equivariant_1968} for the case of a compact group (but still compact base) and, e.g., \cite[p. 21]{zois_18_2010} for the case of non-equivariant bundles over a paracompact base, which can be adapted to the equivariant case by a standard averaging trick.} for $\omega\in\{\operatorname{sm}, \operatorname{cont}\}$ and by Theorem \ref{equivalence sheaves and filtrations} and Proposition \ref{characterisation splittable filtrations}, so does 
\[
\Bun^{\operatorname{alg}}(\C^n,(\C^\times)^n)\longrightarrow \Gra^n_\C. 
\]
Finally, by a general Theorem of Heinzner and Kutzschebauch on equivariant bundles on Stein spaces\footnote{\cite[p. 341]{heinzner_equivariant_1995}, see also \cite[par. 1.4.]{heinzner_geometric_1991} for the notions used in the statement.} the forgetful functor
\[
\Bun^{\operatorname{hol}}(\C^n,(\C^\times)^n)\longrightarrow \Bun^{\operatorname{cont}}(\C^n,\T^n)
\]
induces a bijection on isomorphism classes. Summing up the statements in the previous paragraphs, one obtains:

\begin{prop}
	In the diagram $(\ast)$, all arrows induce bijections on isomorphism classes. The sides of the triangles on the right are equivalences of categories.
\end{prop}

For the functor from algebraic to holomorphic bundles, even more is true:
\begin{prop}\label{toric GAGA affine case}
	The functor 
	\[
	\Bun^{\operatorname{alg}}(\C^n,(\C^\times)^n)\longrightarrow \Bun^{\operatorname{hol}}(\C^n,(\C^\times)^n)
	\]
	is an equivalence of categories.
\end{prop}
\begin{proof}
	We already know that it is essentially surjective by the previous considerations. It is also obviously faithful, so what remains to be checked is that it is full. Consider a map of equivariant holomorphic bundles $\V\longrightarrow \W$. Without loss of generality, we may assume $\V=\widetilde{V},\W=\widetilde{W}$ for some $(\C^\times)^n$-representations  $V,W$. In this case, the statement follows from the following Lemma applied to the bundle $\Hom(\V,\W)$.
	\begin{lem}\label{invariant sections of trivialized bundle}
		Let $V$ be an object of $\Rep^{\operatorname{hol}}((\C^\times)^n)$. For any equivariant section $z\mapsto (z,s(z))$ of the bundle $\widetilde{V}^{\operatorname{hol}}$, there are a finite subset $I\subseteq\Z_{\geq0}^n$ and elements $v_p\in V^{-p}$ for all $p\in I$ s.t.
		\[
		s=\sum_{p\in I}z^pv_p%ToDo plusminus p
		\]
	\end{lem}
	It suffices to check this equality on the dense open subset $U=(\C^{\times})^n\subseteq \C^n$. Since $(\C^\times)^n$ acts simply transitively on $U$, a section is determined by its value $s(1,...,1)=\sum_{p\in I} v_p$. Because the section extends to the whole of $\C^n$, necessarily $I\subseteq\Z^n_{\geq 0}$. This implies the Lemma and consequently the Proposition.
\end{proof}
\begin{rem}
	The other two forgetful functors are still faithful and essentially surjective, but no longer full.
\end{rem}
\begin{rem}
	Using the whole diagram $(\ast)$ and in particular \cite{heinzner_equivariant_1995} to show essential surjectivity of analytification seems to be quite an overkill. It would be interesting to see a direct derivation of the classification of isomorphism classes in $\Bun^{\operatorname{hol}}(\C^n,(\C^\times)^n)$. %In a (failed) attempt to produce one, I came across some nice functional equations for certain classes of matrices of holomorphic functions which can be found in appendix \ref{functional equations}.
\end{rem}
\begin{thm}\label{toric GAGA}
	For any smooth toric variety $X$, analytification yields an equivalence of categories
	\begin{center}
		$\left\{\parbox{4.3cm}{\centering equivariant algebraic vector bundles on $X$} \right\}\longleftrightarrow\left\{ \parbox{3.9cm}{\centering equivariant holomorphic vector bundles on $X^{\operatorname{an}}$} \right\}$
	\end{center}
\end{thm}
\begin{proof}
	For any groups $G,G'$ and $G$-space $X$, one obtains by formal nonsense an identification of the $G\times G'$-equivariant bundles over $X\times G'$ and the $G$-equivariant bundles over $X$, regardless wether one considers this in the algebraic or holomorphic category. In particular, there is an equivalence
	\[
	\Bun^{\omega}(\C^{n}\times (\C^\times)^m,(\C^\times)^{n+m})\longrightarrow\Bun^{\omega}(\C^n,(\C^\times)^n)
	\]
	for $\omega=\operatorname{alg},\operatorname{hol}$ and analytification gives an equivalence of algebraic and holomorphic equivariant bundles over $\C^n\times(\C^\times)^m$. One concludes by noting that every toric variety is covered by open sets equivariantly isomorphic to $\C^n\times (\C^\times)^m$ for some $n,m\in\Z_{\geq 0}$, the intersections of which have again this form.\footnote{See \cite[thm. 3.1.19, thm. 1.3.12, ex. 1.2.21]{cox_toric_2011}}
\end{proof}

\begin{cor}\label{GAGA for connections}
	Let $X$ be a smooth toric variety. Analytification yields an equivalence of categories
	\begin{center}
		$\left\{\parbox{4.3cm}{\centering equivariant algebraic vector bundles with an equivariant connection on $X$} \right\}\longleftrightarrow\left\{ \parbox{3.8cm}{\centering equivariant holomorphic vector bundles with an equivariant connection on $X^{an}$} \right\}.$
	\end{center}
	It restricts to an equivalence of the subcategories of flat connections.
\end{cor}

\begin{proof}
	That the functor is fully faithful follows from the corresponding statement for bundles without a connection. In fact, a map between algebraic vector bundles with connection is compatible with the connections if and only if its analytification is. Similarly, a connection is flat if and only if its analytification is. For essential surjectivity, we can assume that $X=\C^n\times(\C^{\times})^m$ with action by multiplication of $(\C^\times)^{n+m}$ and that a holomorphic bundle $\V$ is given in trivialized form, i.e., as $\V=\pr_{\C^n}^*\widetilde{\V(0)}\otimes\pr_{(\C^\times)^m}^*\Oh_{(\C^{\times})^m}$. One checks that on these (trivial) bundles the canonical connection $d$ is equivariant, so any connection is given as
	\[
	\nabla=d+\Omega,
	\]
	where $\Omega$ is a global invariant holomorphic $1$-form on $\C^n\times(\C^\times)^m$ with values in the vector space ${\End(\V(0))}$. Arguing as in the proof of Lemma \ref{invariant sections of trivialized bundle}, such a form can be written as
	\[
	\Omega=\sum_{p\in\Z^n_{\geq 0}}\sum_{i=1}^{n+m} A_{p,i}z_1^{p_1}\cdot...\cdot z_n^{p_n}\frac{dz_i}{z_i}
	\]
	
	where the $A_{p,i}$ are endomorphisms of $\V(0)$ of multidegree $-p$ (which are taken to be zero if they would cause a pole, i.e., if $p_i=0$). In particular, the connection is algebraic.
\end{proof}
For flat connections, there is also a nicer comparison to representations, which also accounts for the smooth case. Let $\Bun^{\omega}_{\nabla}(\C^n,(\C^\times)^n)$ with $\omega\in\{\operatorname{alg},  \operatorname{hol}\}$ denote the category of $(\C^\times)^n$ algebraic or holomorphic equivariant vector bundles with an equivariant connection and denote by $\Bun^{\omega}_{\nabla^\flat}(\C^n,(\C^\times)^n)$ the respective subcategories of flat connections.
Let $\Bun_{\nabla^\flat}^{\operatorname{sm}}(\C^n,\T^n)$ be the category of smooth equivariant vector bundles with flat equivariant connections. As above, there is a commutative diagram:
\[\tag{$\ast\ast$}
\begin{tikzcd}
\Bun^{\operatorname{alg}}_{\nabla^\flat}(\C^n,(\C^\times)^n)\ar[r]\ar[d]&\Rep^{\operatorname{alg}}((\C^\times)^n)\ar[d]\\
\Bun^{\operatorname{hol}}_{\nabla^\flat}(\C^n,(\C^\times)^n)\ar[r]\ar[d]&\Rep^{\operatorname{hol}}((\C^\times)^n)\ar[d]\\
\Bun^{\operatorname{sm}}_{\nabla^\flat}(\C^n,\T^n)\ar[r]&\Rep^{\operatorname{sm}}(\T^n)
\end{tikzcd}
\]
Here, horizontal arrows are again restriction to $0$ and vertical ones forget about the stronger regularity conditions imposed.
\begin{prop}
	In the diagram $(\ast\ast)$, all arrows are equivalences of categories.
\end{prop}
\begin{proof}
	Since we already know the functors in the right column and the one from algebraic to holomorphic bundles to be an equivalence, it suffices to show that restriction to the fixed point is an equivalence in the holomorphic and smooth cases. We do this in the holomorphic case, the smooth case works the same way.\\
	
	Sending $(V,\nabla)$ to $\ker \nabla$ is an equivalence of categories between flat equivariant connections and equivariant local systems on $\C^n$. Since $\C^n$ is contractible, any local system is necessarily trivial. So restriction from global sections to any point is an isomorphism. In particular, restriction to the fixed point $0$ induces an equivalence of categories with $\Rep^{\operatorname{hol}}((\C^\times)^n)$.\\
	
	For completeness, here is an explicit description of the pseudo-inverse to the restriction functor in $(\ast\ast)$: For any representation $(V,\rho)$ in $\Rep^{\operatorname{hol}}((\C^\times)^n)$, as before, consider the bundle $\widetilde{V}$ with product action. Its sheaf of sections $V\otimes_\C\Oh_{\C^n}$ is equipped with the canonical equivariant connection $d$ given by 
	\[
	d(v\otimes f)=v\otimes df
	\]
	and sending a map $f:(V,\rho_V)\longrightarrow (W,\rho_W)$ of representations to $f\otimes\Id$ this defines the pseudo-inverse
	\[
	\Rep^{\operatorname{hol}}((\C^\times)^n)\longrightarrow \Bun_{\nabla^\flat}(\C^n,(\C^\times)^n).
	\]
\end{proof}
\begin{rem}
An early impetus for the questions treated in this section came from studying Kapranov's proof \cite{kapranov_real_2012} of the equivalence between (complex) Mixed Hodge Structures and algebraic toric vector bundles with an equivariant connection on $\C^2$. A crucial step in the proof consists in an application of the equivariant Radon-Penrose transform, that yields an equivalence between equivariant holomorphic vector bundles with a connection on $\C^2$ and a certain holomorphic toric vector bundle on $\Pro^2_\C\setminus\{[1,0,0]\}$. The latter are then related to triples of opposed filtrations via the Rees-bundle construction. One then has to check that if a bundle on one side of the equivalence is algebraic, so is its counterpart on the other side. By Theorem \ref{toric GAGA} and Corollary \ref{GAGA for connections} this last step is in a certain sense redundant: The algebraic and holomorphic categories on both sides are equivalent.
\end{rem}
\section{The Fr\"olicher approximating bundle}
Denote by $(C_{\Cdot}^\infty(X,\C),d)$ the complex of $\C$-vector spaces given by complex-valued differential forms and exterior differentiation. This complex carries a filtration given by 
\[
F^pC_{\Cdot}^\infty(X,\C):=\bigoplus_{\substack{r+s=\Cdot\\r\geq p}}C_{p,q}^\infty(X,\C)
\]
It is a well-known result that this induces the so-called Fr\"olicher spectral sequence, of the form
\[
E_1^{p,q}(X)=H^q(X,\Omega^p)\Longrightarrow H_{dR}^{p+q}(X,\C).
\]

In \cite{popovici_adiabatic_2019} (from which we adopt some notations in this section), Dan Popovici associates with every compact connected complex manifold $X$  a holomorphic vector bundle, called the \textbf{Fr\"olicher approximating vector bundle} (short FAVB), denoted $\cA^k$ on $\C$ together with distinguished isomorphisms of the fibres:

\[\tag{$\ast$}
\psi_h:\cA^k_h\cong \begin{cases}H^k_{dR}(X,\C)&\text{if }h\neq 0\\\bigoplus_{p+q=k}E_{\infty}^{p,q}(X) &\text{if } h=0\end{cases}
\]

Here, $E_\infty(X)$ denotes the limiting page of the Fr\"olicher spectral sequence of $X$. Recall that convergence of the spectral sequence means there the separated and exhaustive descending filtration induced by $F$ on $H_{dR}^k(X,\C)$, also denoted $F$ and called the Hodge-filtration, satisfies $E^{p,q}_{\infty}=\gr_F^{p}H_{dR}^{p+q}$. Further, $E_\infty^{p,q}\cong E_r^{p,q}$ for any $r$ with the property that all differentials entering or leaving the bidegree $(p,q)$ are zero from page $E_r$ onward. In particular, one could replace $\infty$ in the statement with the minimal $r_0$ s.t. the spectral sequence degenerates (e.g. one always has $r_0\leq\dim_\C X$).\\

Popovici's construction depends on the choice of a metric  on $X$. Let us briefly recall it: The main point consists in contructing a $C^\infty$-family of Laplace-type pseudo-differential operators $(\tilde{\Delta}_h)_{h\in\C}$ on the space of $k$-forms $C^\infty_k(X,\C)$. This implies there kernels form a $C^\infty$-vector bundle over $\C$. Then one computes that inclusion and projection induce isomorphisms
\[\tag{$\ast\ast$}
\ker\tilde{\Delta}_h\cong\begin{cases}H_{d_h}^k(X,\C)&\text{if }h\neq 0\\\bigoplus_{p+q=k}E_{\infty}^{p,q}(X)&\text{if }h=0.\end{cases}
\]
Here, $H_{d_h}(X,\C)$ denotes cohomology with respect to the `twisted' differential $d_h=h\del+\delbar$. For every $0\neq h\in\C$, the map
\begin{align*}
\theta_h:C^\infty_k(X,\C)&\longrightarrow C^\infty_k(X,\C)\\
\sum_{p+q=k}\omega^{p,q}&\longmapsto \sum_{p+q=k} h^p\omega^{p,q}
\end{align*}
induces isomorphisms $H_{dR}^k(X,\C)\cong H^k_{d_h}(X,\C)$. This varies holomorphically in $h$, hence the bundle $\ker\tilde{\Delta}$ is holomorphic over $\C^\times$ (even a constant local system). Therefore, it is necessarily also holomorphic over the whole of $\C$ by Riemann's removable singularity theorem.\\

We show now that $\cA^k$ can be identified with a Rees-bundle. Slightly abusing notation, we identify $\cA^k$ with its sheaf of sections.

\begin{thm}\label{Popovici and Rees}
Let $X$ be a compact connected complex manifold and $\cA^k$ the $k$-th Fr\"olicher approximating bundle. There is a canonical isomorphism
\[
\cA^k\cong \xi_{\A^1}(H_{dR}^k(X,\C),F)^{\an}
\]
where $F$ denotes the Hodge filtration and the superscript $^{\an}$ means analytification.
\end{thm}
\begin{cor}
The bundle $\cA^k$ carries a (necessarily unique) algebraic and $\C^\times$-equivariant structure which is independent of $g$. The association $X\longmapsto \cA^k$ is contravariantly functorial for maps between (compact) complex manifolds.
\end{cor}
Note that the Rees-bundle construction makes perfect sense for infinite vector spaces, yielding a quasi-coherent (and for one filtration or two filtrations: locally free) sheaf. Hence, one can also omit the compactness condition and use $\xi_{\A^1}(H_{dR}^k(X,\C),F)$ as a definition of the FAVB for arbitrary complex manifolds (although one may argue that the filtration used here is not the `right' filtration to consider in the non-compact case).\\

By Lemma \ref{Restriction properties of the affine Rees bundle}, the bundle $\xi_{\A^1}(H_{dR}^k(X,\C),F)$ has the desired identifications $(\ast)$. The isomorphism with $\cA^k$ will follow from the following general `base-change-property' for Rees-bundles.

\begin{lem}\label{base-change}
Given a complex manifold $X$ and $k\in\Z$, consider the complex $(A_\Cdot,d_\xi)$ of sheaves on $\C$ defined by  $A_\Cdot=\xi_{\A^1}(C^\infty_{\Cdot}(X,\C),F)$ and $d_\xi=\xi_{\A^1}(d)$. There is a canonical identification of toric vector bundles on $\C$:
\[
H^k\big(A_\Cdot^{\an},d_\xi^{\an}\big)\operatorname{ mod }T \cong \xi_{\A^1}(H_{dR}^k(X,\C),F)^{\an}
\]
where $T$ denotes the torsion subsheaf.\\

There are canonical isomorphisms
\begin{align*}
\alpha_k:A_k\longrightarrow & ~\widetilde{C_k^\infty(X,\C)\otimes\C[z]}\\
\sum_{p+q=k}\omega^{p,q}\cdot p(z^{\pm 1})\longmapsto& \sum_{p+q=k}\omega^{p,q}p(z^{\pm 1})z^p
\end{align*} 
under which $d_\xi$ gets identified with $d_z=z\cdot(\del\otimes\Id)+(\delbar\otimes\Id)$.
\end{lem}
\begin{proof}
The first part is a direct application of Proposition \ref{maps between filtered spaces vs. maps between Rees-bundles}, i.e. the fact that the Rees-construction commutes with kernels and commutes with cokernels up to torsion.
The isomorphism in the second part is the trivialization from Proposition \ref{characterisation splittable filtrations} (note that there is the canonical splitting of $F$ given by $C_k^\infty(X,\C)=\bigoplus_{p+q=k}C_{p+q}^\infty(X,\C)$).
\end{proof}
\begin{proof}[Proof of Theorem \ref{Popovici and Rees}]
Consider the holomorphic bundle $\ker\tilde{\Delta}_z$ as a subsheaf of $C_k^\infty(X,\C)\otimes_\C\Oh_{\C}$. By equation $(\ast\ast)$, it is contained in $\ker d_z\cap C_k^\infty(X,\C)\otimes_\C\Oh_{\C}\cong A_k\cap \ker d_\xi$ and therefore projects to $H^k\big(A_\Cdot^{\an},d_\xi^{\an}\big)\operatorname{ mod }T$. The result now follows from $(\ast)$ and Lemma \ref{Restriction properties of the affine Rees bundle}.
\end{proof}

We end this section and the article by sketching a few open ends and questions:

\begin{itemize}
\item \textbf{(Second filtration)} As we have seen, the FAVB has the metric independent definition as $\xi_{\A^1}(H_{dR}^k(X,\C),F)$. One minor disadvantage of this is that the resulting bundle does not see the real structure on $H_{dR}^k(X,\C)$. One could also take the conjugate Hodge-filtration into account and obtain a bundle on $\A^2$, given as $\xi_{\A^2}(H_{dR}^k(X,\C),F,\bar{F})$, which is now also equivariant with respect $\G_m^2$ and the antilinear involution on the base given by $(z_1,z_2)\mapsto (\bar{z}_2,\bar{z}_1)$ and hence descends to a real bundle on $\A^2_\R$. Remaining with the bundle on $\A^2$, its restriction to $\A^1\times\{h\}$ ($h\neq 0$) yields the FAVB, while its restriction to $\{h\}\times\A^1$ ($h\neq 0$) yields the analogous bundle for the conjugate spectral sequence. It is, similarly to Proposition \ref{base-change}, isomorphic to the cohomology (modulo Torsion) of the de-Rham complex with parameters $z_1,z_2$ and deformed differential, now given by $z_1\del+z_2\delbar$. The fibre of $\xi_{\A^2}(H_{dR}^k(X,\C),F,\bar{F})$ over $(0,0)$ is given by the space \[D=\bigoplus_{p,q\in\Z}\frac{F^p\cap\bar{F}^q}{F^{p+1}\cap\bar{F}^{q}+F^{p}\cap\bar{F}^{q+1}}\]
which measures the defect of the de-Rham cohomology being pure, i.e. $H_{dR}^k(X,\C)\cong \bigoplus_{p+q=k} F^p\cap \bar{F}^q$. Can one also describe this bundle, arising from two filtrations, as the kernel of a family of differential operators? In particular, is there a harmonic theory for (the bigraded components of) the space $D$?\\
One can also form $\xi_{\Pro^1}(H_{dR}^k(X,\C),F,\bar{F})$, i.e. descend the bundle to projective space. If the Hodge-filtrations induce a pure Hodge structure of some weight $k'$ (not necessarily $=k$) on $k$, this is a Twistor Structure in the sense of Simpson (after forgetting most of the action).

\item \textbf{(Relative version)} In \cite{popovici_adiabatic_2019} Popovici also considers the situation of a family of compact complex manifolds i.e. a proper holomorphic submersion $\pi:X\longrightarrow B$ and constructs a bundle $\cA^k$ on $\C\times B$, s.t. the restriction to each slice $\C\times \{b\}$ is the FAVB for $X_b=\pi^{-1}(b)$. Is there also a purely algebraic construction for this bundle? To make the Rees-bundle construction work in the relative setting, one maybe should answer the following question: Are the $F^pH^k_{dR}(X_b,\C)$ fibres of a coherent subsheaf of the (flat) holomorphic vector bundle $R^k\pi_*\C$?\\
It would also be interesting to consider the two-filtration version above in the relative setting.
\end{itemize}

\textbf{Acknowledgments:} This article was written at LMU M\"unchen. Part of the material is also contained in my PhD-thesis, elaborated at WWU M\"unster with financial support by the SFB 878. I warmly thank Christopher Deninger, J\"org Sch\"urmann and Dan Popovici for interesting and encouraging discussions and advice.
\bibliographystyle{acm}
\bibliography{toric}

\begin{thebibliography}{10}

\bibitem{atiyah_$k$-theory_1967}
{\sc Atiyah, M.~F.}
\newblock {\em $K$-theory}.
\newblock {W.A. Benjamin, Inc.}, 1967.

\bibitem{cox_toric_2011}
{\sc Cox, D.~A., Little, J.~B., and Schenck, H.~K.}
\newblock {\em Toric {{Varieties}}}.
\newblock No.~134 in Graduate {{Studies}} in {{Mathematics}}. {American
  Mathematical Soc.}, 2011.

\bibitem{deninger_$gamma$-factors_1991}
{\sc Deninger, C.}
\newblock On the $\gamma$-factors attached to motives.
\newblock {\em Inventiones mathematicae 104}, 2 (1991), 245--262.

\bibitem{deninger_$gamma$-factors_2001}
{\sc Deninger, C.}
\newblock On the $\gamma$-factors of motives ii.
\newblock {\em Documenta mathematica Vol. 6\/} (2001), 69--97.

\bibitem{goncharov_hodge_2008}
{\sc Goncharov, A.~B.}
\newblock Hodge correlators.
\newblock {\em arXiv:0803.0297 [math]\/} (Mar. 2008).

\bibitem{heinzner_geometric_1991}
{\sc Heinzner, P.}
\newblock {Geometric invariant theory on Stein spaces.}
\newblock {\em Mathematische Annalen 289}, 4 (1991), 631--662.

\bibitem{heinzner_equivariant_1995}
{\sc Heinzner, P., and Kutzschebauch, F.}
\newblock An equivariant version of {{Grauert}}'s {{Oka}} principle.
\newblock {\em Inventiones mathematicae 119}, 1 (1995), 317--346.

\bibitem{kaneyama_equivariant_1975}
{\sc Kaneyama, T.}
\newblock On equivariant vector bundles on an almost homogeneous variety.
\newblock {\em Nagoya Mathematical Journal 57\/} (1975), 65--86.

\bibitem{kaneyama_torus-equivariant_1988}
{\sc Kaneyama, T.}
\newblock Torus-equivariant vector bundles on projective spaces.
\newblock {\em Nagoya Mathematical Journal 111\/} (1988), 25--40.

\bibitem{kapranov_real_2012}
{\sc Kapranov, M.}
\newblock Real mixed {{Hodge}} structures.
\newblock {\em Journal of Noncommutative Geometry 6}, 2 (2012), 321--342.

\bibitem{klyachko_vector_1991}
{\sc Klyachko, A.}
\newblock Vector bundles and torsion free sheaves on the projective plane.
\newblock {\em MPIM Preprint\/} (1991).

\bibitem{klyachko_equivariant_1990}
{\sc Klyachko, A.~A.}
\newblock Equivariant bundles on toral varieties.
\newblock {\em Mathematics of the USSR-Izvestiya 35}, 2 (1990), 337.

\bibitem{penacchio_structures_2003}
{\sc Penacchio, O.}
\newblock {\em Structures de Hodge mixtes et fibrés sur le plan projectif
  complexe}.
\newblock PhD thesis, July 2003.

\bibitem{penacchio_mixed_2011}
{\sc Penacchio, O.}
\newblock Mixed {{Hodge}} structures and equivariant sheaves on the projective
  plane.
\newblock {\em Mathematische Nachrichten 284}, 4 (Mar. 2011), 526--542.

\bibitem{perling_graded_2004}
{\sc Perling, M.}
\newblock Graded rings and equivariant sheaves on toric varieties.
\newblock {\em Mathematische Nachrichten 263-264}, 1 (2004), 181--197.

\bibitem{popovici_adiabatic_2019}
{\sc Popovici, D.}
\newblock Adiabatic {{Limit}} and {{Deformations}} of {{Complex Structures}}.
\newblock {\em arXiv:1901.04087 [math]\/} (Jan. 2019).

\bibitem{segal_equivariant_1968}
{\sc Segal, G.}
\newblock Equivariant $k$-theory.
\newblock {\em Publications Mathématiques de l'IHÉS 34\/} (1968), 129--151.

\bibitem{simpson_hodge_1997}
{\sc Simpson, C.}
\newblock The {H}odge filtration on nonabelian cohomology.
\newblock {\em Algebraic geometry---{S}anta {C}ruz 1995 62\/} (1997), 217--281.

\bibitem{simpson_mixed_1997}
{\sc Simpson, C.}
\newblock Mixed twistor structures.
\newblock {\em arXiv:alg-geom/9705006\/} (May 1997).

\bibitem{zois_18_2010}
{\sc Zois, I.~P.}
\newblock 18 {{Lectures}} on {{K}}-{{Theory}}.
\newblock {\em arXiv:1008.1346 [math]\/} (Aug. 2010).

\end{thebibliography}
\end{document}